\documentclass{elsarticle}
\usepackage{amsmath,amsthm}
\usepackage{amsfonts}
\usepackage{amsmath}
\usepackage{amssymb}
\usepackage{graphicx}
\usepackage{wasysym}
\usepackage{bm}
\usepackage{dictsym}
\usepackage{cypriot}
\usepackage{manfnt}
\usepackage{staves}
\usepackage{protosem}
\usepackage{color}
\usepackage{pifont}
\usepackage{tipa}
\usepackage{phaistos}
\usepackage{fourier-orns}
\usepackage{MnSymbol}
\usepackage{amssymb}
\usepackage{subfigure}
\usepackage{braket}
\usepackage{bm}
\theoremstyle{plain}
\newtheorem{thm}{Theorem}[section]
\newtheorem{lem}{Lemma}[section]

\newtheorem{cor}{Corollary}
\newtheorem*{obs}{Remark}
\theoremstyle{definition}
\newtheorem{defn}{Definition}[section]

\theoremstyle{remark}

\usepackage{tikz}
\newcommand*{\id}{{\normalfont\hbox{1\kern-0.15em \vrule width .8pt depth-.5pt}}}

\usepackage{natbib}

\makeatletter
\def\@author#1{\g@addto@macro\elsauthors{\normalsize%
    \def\baselinestretch{1}%
    \upshape\authorsep#1\unskip\textsuperscript{%
      \ifx\@fnmark\@empty\else\unskip\sep\@fnmark\let\sep=,\fi
      \ifx\@corref\@empty\else\unskip\sep\@corref\let\sep=,\fi
      }%
    \def\authorsep{\unskip,\space}%
    \global\let\@fnmark\@empty
    \global\let\@corref\@empty  
    \global\let\sep\@empty}%
    \@eadauthor={#1}
}
\makeatother

\begin{document}

\begin{frontmatter}

\title{On the multiplicity of Laplacian eigenvalues and Fiedler partitions}

\author[Aq]{Eleonora Andreotti\corref{cor1}}
\ead{eleonora.andreotti@graduate.univaq.it}
\cortext[cor1]{Corresponding author}

\author[Bolo,INFN]{Daniel Remondini}
\ead{daniel.remondini@unibo.it}

\author[Bolo]{Graziano Servizi}

\author[Bolo,INFN]{Armando Bazzani}
\ead{armando.bazzani@unibo.it}
\address[Aq]{Department of Information Engineering, Computer Science and Mathematics (DISIM), University of L'Aquila, 67100 L'Aquila, Italy}
\address[Bolo]{Department of Physics and Astronomy (DIFA), University of Bologna, 40127 Bologna, Italy}

\address[INFN]{INFN Section of Bologna, Italy}

\begin{abstract}
In this paper we investigate the relation between eigenvalue distribution
and graph structure of two classes of graphs: the $(m,k)$-stars and $l$-dependent
graphs. We give conditions on the topology and edge weights in order to get values and multiplicities of Laplacian matrix
eigenvalues. We prove that a vertex set reduction on graphs with
$(m,k)$-star subgraphs is feasible, keeping the same eigenvalues with reduced multiplicity. Moreover, some useful eigenvectors
properties are derived up to a product with a suitable matrix.
Finally, we relate these results with Fiedler spectral partitioning
of the graph and the physical relevance of the results is shortly
discussed.
\end{abstract}
\end{frontmatter}


\section{Introduction}

In the context of complex networks, the Laplacian formalism can be used to find many useful properties of the underlying graph
\cite{MERRIS1994143, Chung97, benzi, biggs93, MR1324340, Mitrovic2009}.
In particular, the idea of spectral clustering is to extract some important information on the network structure from the
matrices associated with the network, by considering one or few of the leading eigenvectors \cite{Boccaletti2006175}. \\
According to the Fiedler theory, a bipartition of a graph can be obtained from the second eigenvector both of the
Laplacian matrix \cite{fiedler73, fiedler75, donath1973lower}, and of the Normalized Laplacian matrix \cite{Chung97}.
More precisely, one can obtain a good ratio cut of the graph from any vector orthogonal to the all-ones vector, with a small Rayleigh quotient \cite{conf/focs/Mihail89}.\\
In general, a different number of clusters can obtained by means of the following strategies:
\begin{description}
\item[a)]  by a Recursive Spectral Bisection (RSB) \cite{journals/concurrency/BarnardS94, I91partitioningof, CPE:CPE4330070103}: after using the Fiedler
eigenvector to split the graph into two subgraphs, one can find the Fiedler eigenvector in each of these subgraphs, and continue recursively until
some a-priori criterion is satisfied;
\item[b)] by using the first $k$  eigenvectors related to the smallest eigenvalues, to induce further partitions through clustering algorithms applied
to the corresponding invariant subspace \cite{journals/dam/AlpertKY99, journals/tcad/ChanSZ94}.
\end{description}

We consider the second approach, recalling that the optimal number $k$ of clusters is often indicated by a large gap between the $k$ and the $k+1$ eigenvalues
for both the Laplacian and Normalized Laplacian matrices \cite{Lee:2012:MSP:2213977.2214078}.\\
Within this framework, we are interested consider the algebraic multiplicity of Laplacian eigenvalues, since the corresponding eigenvectors
can be considered equivalent in a partition procedure of graphs. In presence of multiple eigenvalues, we investigate the possibility of reducing the dimensionality
of the original graph (i.e. of removing some of its nodes) keeping fixed its spectral properties
\cite{onred11, DBLP:conf/cdc/BeckLLW09, DBLP:conf/cdc/DengMM09,sonin1999state, Sadiq:2000:APM:344358.344369}.

After some preliminary remarks (section \ref{sec:2}), in section \ref{sec:3} we define two classes of graphs, by giving conditions on the graph structure
which implies the presence of multiple eigenvalues. Then we propose a reduction on the number of nodes, such that it is possible to get an identical spectrum
for the Laplacian matrices of the original and the reduced graphs (up to the eigenvalue multiplicity) with respect to a suitable diagonal
\textit{mass matrix}, that changes the link weights a plays the role of metric matrix.
Furthermore, we get a connection between the primary and the reduced graph eigenvectors.
Thanks to these results it is possible to perform a partition of the primary and the reduced graphs using the same procedure.
Finally, in section \ref{sec:4} we draw some conclusions and give an outlooks on future developments.\\

\section{Premises}\label{sec:2}
We consider an undirected weighted connected graph $\mathcal G:=(\mathcal V, \mathcal E, w)$, where the $n$ vertices $\mathcal V$ are connected by
the $\mathcal E$ edges with $w$ the weight function: $w:\mathcal E\rightarrow \mathbb R^{+}.$
Let $A$ be the weighted adjacency matrix, which is symmetric since the graph is undirected ($A\in Sym_n(\mathbb R^+)$),
$$A_{ij} = \begin{cases} w(i,j), & \mbox{if $i$ is connected to $j$ } (i\sim j) \\
0 & \mbox{otherwise }  \end{cases}$$
where $i,j\in\mathcal V,$,
the Laplacian matrix $L\in Sym_n (\mathbb R)$ and normalized Laplacian matrix $\hat L\in Sym_n (\mathbb R)$  are respectively defined
$$L_{ij} = \begin{cases} -w(i,j), & \mbox{if } i\sim j \\
\sum_{k=1}^n w(i,k), & \mbox{if } i= j\\
0 & \mbox{otherwise }  \end{cases}$$

$$\hat L_{ij} = \begin{cases} -\displaystyle \frac{w(i,j)}{\sqrt{\sum_{k=1}^n w(i,k)\sum_{k=1}^n w(k,j)}}, & \mbox{if } i\sim j \\
1, & \mbox{if } i= j\\
0 & \mbox{otherwise }.  \end{cases}$$

Whenever we refer to the $k$-th eigenvalue of a Laplacian matrix, we will refer to the $k$-th nonzero eigenvalue according to a increasing order.
For the classical results on Laplacian matrices theory, one may refer to \cite{Chung97,doi:10.1080/03081088508817681, MERRIS1994143}.

\section{Eigenvalues multiplicity theorems}\label{sec:3}
The first result is an extension of Theorem {(4)} in \cite{grone94} to weighted graphs: by defining the weighted $(m,k)$-stars in a graph,
we are able to give a condition on both the structure and edge weights of graphs in order to get the eigenvalue multiplicity.
{As we will see later, an $(m,k)-$star is nothing else that the union of a $k$-cluster of order $m$ and its $k$ neighbours.}\\
The second result, that is the main results of this work, is a further extension of the previous Theorem to understand the relation between
eigenvalue multiplicity and the structure of the weights of graphs.\\
The third result concerns the reduction of graphs with one or more $(m,k)$-stars under some conditions, and possible applications
on spectral graphs partitioning.

\subsection{$(m,k)$-star and $l$-dependent: eigenvalues multiplicity}
We recall that a vertex of a graph is said \textit{pendant} if it has exactly one neighbour, and \textit{quasi pendant}
if it is adjacent to a pendant vertex. It is possible to prove that the multiplicity $m_{L}(1)$ of the eigenvalue $\lambda=1$ of the Laplacian of an
unweighted graph, is greater or equal than the number of pendant vertices less the number of quasi pendant vertices of the graph \cite{FARIA1985255}.\\
To extend these definitions to vertices with $k$ neighbours, we define a $(m,k)$-star:

\begin{figure}[!!h]
\begin{subfigure}{}
\includegraphics[width=6cm]{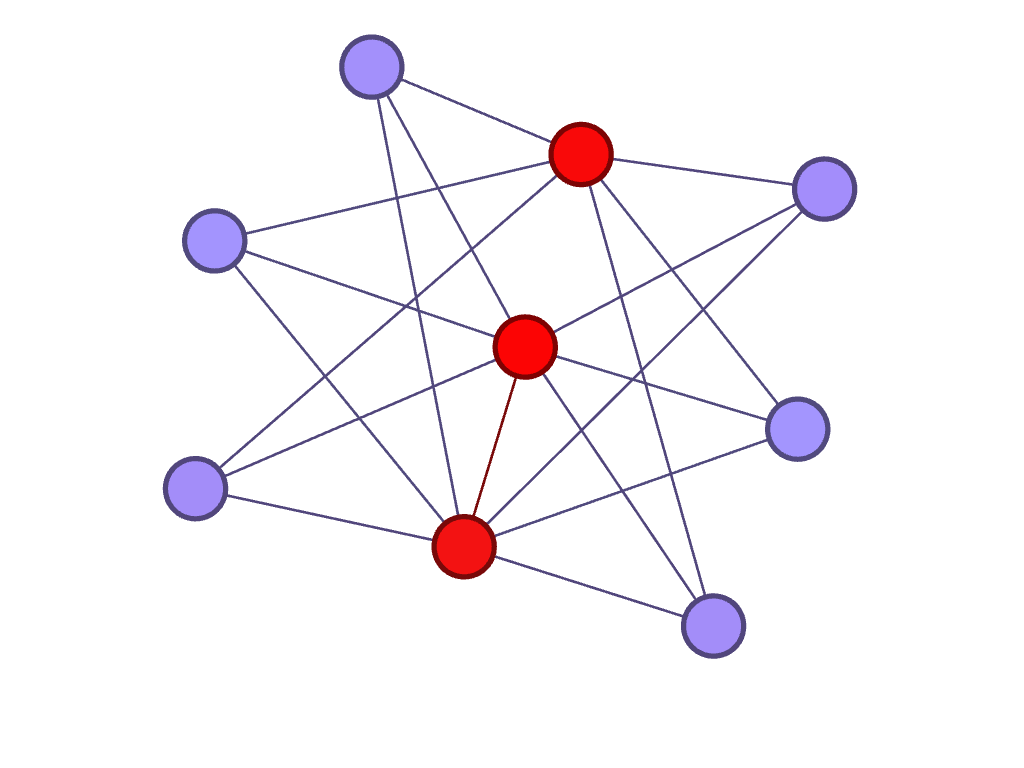}\end{subfigure}
\begin{subfigure}{}
\includegraphics[width=6cm]{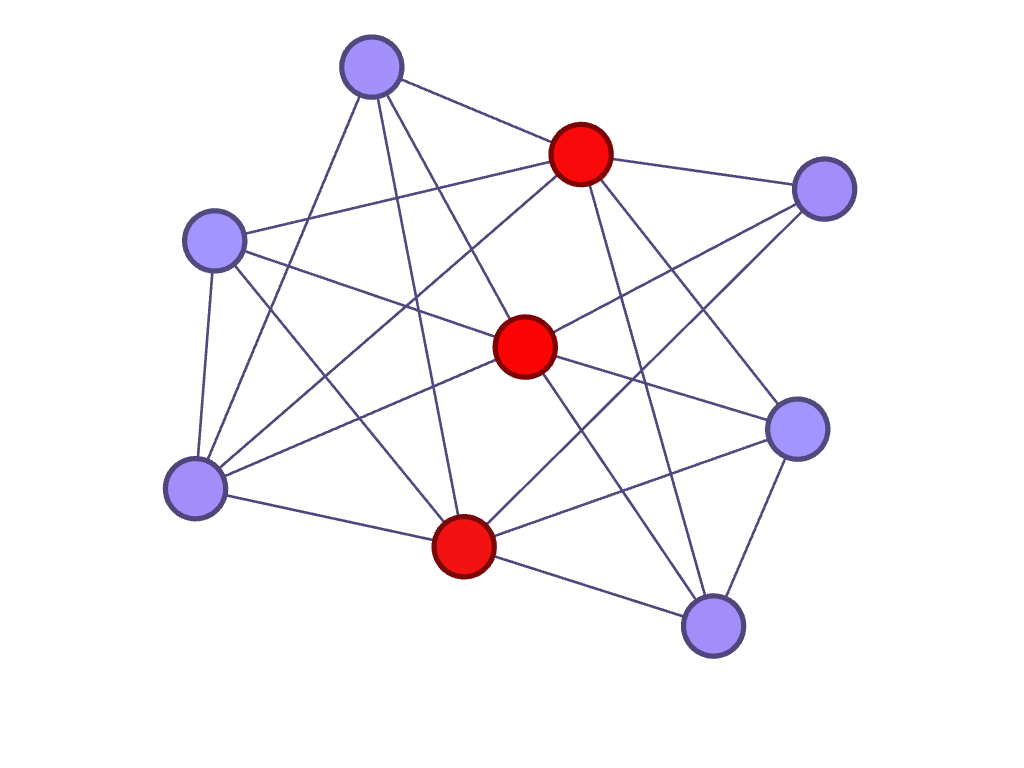}\end{subfigure}
\caption{In the left a $S_{6,3}$ graph, in the right a $S_{3,6}$ graph.}
\end{figure}

\begin{defn}[$(m,k)$-star: $S_{m,k}$ ]
A $(m,k)$-star is a graph $\mathcal G=(\mathcal V, \mathcal E,w)$ whose vertex set $\mathcal V$ has a bipartition $(\mathcal V_1,\mathcal V_2)$
of cardinalities $m$ and $k$ respectively, such that the vertices in $\mathcal V_1$ have no connections among them, and each of these vertices is connected
with all the vertices in $\mathcal V_2$: i.e
$$\forall i\in \mathcal V_1,\forall j\in \mathcal V_2,\quad (i,j)\in \mathcal E$$
$$\forall i,j\in \mathcal V_1, \quad (i,j)\notin \mathcal E$$

We denote a $(m,k)$-star graph with partitions of cardinatilty $|\mathcal V_1|=m$ and $|\mathcal V_2|=k$ by $S_{m,k}.$
\end{defn}

We define a \textit{$(m,k)$-star of a graph} $\mathcal G=(\mathcal V, \mathcal E,w)$ as the $(m,k)$-star of partitions $\mathcal V_1$, $\mathcal V_2\subset \mathcal V$,
both of them univocally determined, such that the vertices in $\mathcal V_1$ have no connection with vertices
in $\mathcal V\setminus \mathcal V_2$ in $\mathcal G., $: i.e.\\
$$\forall i\in \mathcal V_1,\forall j\in \mathcal V_2,\quad (i,j)\in \mathcal E$$
$$\forall i \in \mathcal V_1, \forall j\in \mathcal V\setminus\mathcal V_2 \quad (i,j)\notin \mathcal E$$

{ \begin{obs} In \cite{grone94} is defined a $k$-cluster of $\mathcal G$ to be an independent set of $m$ vertices of $\mathcal G$, $m>1$,
each of which with the same set of neighbours.
The order of a $k$-cluster is the number of vertices in $k$-cluster.
Therefore, the set $\mathcal V_1$ of the $(m,k)$-star is a $k$-cluster of order $m$ and the set $\mathcal V_2$ is the set of the $k$ neighbour vertices.
An $(m,k)-$star of a graph $\mathcal G$ is the union of a $k$-cluster (i.e. $\mathcal V_1$) and its neighbour vertices (i.e. $\mathcal V_2$).
\end{obs}}


By defining the degree and weight of a $(m,k)$-star we simplify the stating of the theorems on eigenvalues multiplicity.

\begin{defn}[Degree of a $(m,k)$-star: $deg(S_{m,k})$]
The \textit{degree} of a $(m,k)$- star is $deg(S_{m,k}):=m-1$ and the degree of a set $\mathcal S$ of $(m,k)$-stars, as $m$ and $k$ vary in $\mathbb N$ ,
such that $|\mathcal S|=l,$ is defined as the sum over each $(m,k)$-star degree, i.e.
$$deg(\mathcal S):=\sum_{i=1}^l deg(S_{m_i,k_i}).$$
\end{defn}

\begin{defn}[Weight of a $(m,k)$-star: $w(S_{m,k})$]
The \textit{weight} of a $(m,k)$-star of vertices set $\mathcal V_1\cup\mathcal V_2$ is defined as the strength of the vertices in $\mathcal V_1$,
provided that the following condition holds:\\
let $\{i_1,...,i_m\}=\mathcal V_1$, then $w(i_1,j)=...=w(i_m,j), \forall j\in\mathcal V_2.$. More precisely the
weight of a $(m,k)$-star: $w(S_{m,k})$ is $$w(S_{m,k}):=\sum_{j\in \mathcal V_2}w(i,j)\mbox{ for any }i\in\mathcal V_1.$$

+\end{defn}

We are ready to enunciate the first theorem, that is an extension to weighted graph of the theorem in \cite{grone94}.
Given a graph $\mathcal G=(\mathcal V,\mathcal E,w)$ associated with the Laplacian matrix $L$, and denoting $\sigma(L)$ the set of the
eigenvalues of $L$ and $m_L(\lambda)$ the algebraic multiplicity of the eigenvalue $\lambda$ in $L$, the following theorem holds\\
\begin{thm}\label{th:one}
Let
\begin{itemize}
\item $s$ be the number of all the $S_{m,k}$ as $m$ and $k$ vary in $\mathbb N$  and $m+k\leq n,$ of $\mathcal G$;
\item  $r$ be the number of $S_{m,k}$ with different weight, $w_1,...,w_r$, i.e. $w_i\neq w_j$ for each $i\neq j,$ where $ i,j\in\{1,...,r\};$\\
\end{itemize}
then for any $i\in\{1,...,r\},$
$$\exists \lambda\in{\sigma(L)} \mbox{ such that } \lambda=w_i \mbox{ and } m_{L}(\lambda)\geq deg(\mathcal S_{w_i})$$
where $\mathcal S_{w_i}:=\{S_{m,k}\in \mathcal G | w(S_{m,k})=w_i\}$.
\end{thm}
{
Before proving Theorem \ref{th:one}, we introduce some useful definitions.

\begin{defn}[$k$-pendant vertex]
A vertex of a graph is said to be \textit{$k$-pendant} if its neighborhood contains exactly $k$ vertices.
\end{defn}

\begin{defn}[$k$-quasi pendant vertex]
A vertex of a graph is said to be \textit{$k$-quasi pendant} if it is adjacent to a $k$-pendant vertex.
\end{defn}
We remark that in the definition of an $(m,k)-$star, the vertices in $\mathcal V_1$ are $k-$pendant vertices, and
vertices in $\mathcal V_2$ are $k-$quasi pendant vertices.\\

\begin{proof}\ref{th:one} \\
We consider connected graphs; indeed if a graph is not connected the same result holds, since the $(m,k)$-star degree
of the graph is the sum of the star degrees of the connected components and
the characteristic polynomial of $L$ is the product of the characteristic polynomials of the connected components. \\

Let a $(m,k)$-star of the graph $\mathcal G$.\\
Under a suitable permutation of the rows and columns of weighted adjacency matrix $A$, we can label the $k$-pendant vertices with the indices $1,...,m$,
and with $m+1,...,m+k$ the indices of the $k$-quasi pendant vertices.\\
We call $v_1,...,v_m$ the rows corresponding to $k$-pendant vertices, then the adjacency matrix has the following form
\[ A = \left( \begin{array}{ccc|ccccccc}
 0 & ... & 0 &  w(1,m+1) & w(1,m+2) &...& w(1,m+k)& 0&...&0\\
 \vdots & ... & \vdots &  \vdots & \vdots &...& \vdots & 0&...&0\\
 0  & ... & 0 &  w(m,m+1) & w(m,m+2) &...& w(m,m+k)& 0&...&0\\
\hline
w(1,m+1)& ... &  w(m,m+1)  &  &  & & & & &\\
  \vdots & ... & \vdots &  &  & & & & &\\
 w(1,m+k)& ... &  w(m,m+k) &  &  & & & & & \\
 0& ... & 0 &  &  & & & & & \\
 \vdots& ... & \vdots &  &  & A_{22} & & & \\
 0& ... & 0 &  &  & & & & & \\
\end{array} \right)\]
where the block $A_{22}$ is any $(n-m)\times(n-m)$ symmetric matrix.\\ 
The $m$ rows (and $m$ columns) $v_1,...,v_m$ are linearly dependent such that $v_1=...=v_m$, then $v_1,...,v_{m-1}\in ker(A)$.\\
Hence
$$\exists \mu_1,...,\mu_{m-1}\in\sigma(A)\quad \mbox{ such that }\quad \mu_1=...=\mu_{m-1}=0.$$

By considering the Laplacian matrix $L$, it has at least $m$ diagonal entries with value $\sum_{j=1}^k w(1,m+j)=w(S_{m,k}):=w_1$.\\

Then also in the matrix $(L-w_1 I)$ there are the linearly dependent vectors $v_i, \ i\in\{1,...,m\}$, hence  $v_1,...,v_{m-1}\in ker(L-w_1 I)$ and
$$\exists \mu_1,...,\mu_{m-1}\in\sigma(L-w_1 I)\quad \mbox{ such that }\quad \mu_1=...=\mu_{m-1}=0.$$

Let $\mu_i$ be one of these eigenvalues, then

$$0=det((L-w_1 I)-\mu_i I)=det(L-(w_1 +\mu_i )I)$$

so that $\lambda:=w_1\in\sigma (L)$ with multiplicity greater or equal to $deg(S_{m,k})$.\\

Let us now consider a number $s$ of $S_{m,k}$ in $\mathcal G$, namely $S_{m_1,k_1},...,S_{m_s,k_s}$.
Denoting $w_1,...,w_r$ the different weights of such a $(m,k)$-stars, and $r\leq s$,
we prove that for any $i\in\{1,...,r\},$
$$\exists \lambda\in{\sigma(L)} \mbox{ such that } \lambda=w_i \mbox{ and the multiplicity of } \lambda\geq deg(\mathcal S_{w_i})=
\sum_{S_{m_j,k_j}\in\mathcal S_{w_i}} deg(S_{m_j,k_j}),$$
where $\mathcal S_{w_i}:=\{S_{m,k}\in \mathcal G | w(S_{m,k})=w_i\}$.

Let $i\in\{1,...,r\}$ and let $R_i\leq r$ be the number of $(m,k)$-stars in $\mathcal  S_{w_i}$, and $\sum_{i=1}^r R_r=s$,
we assume that the first $R_1$ indexes refer to the $(m,k)$-stars in $\mathcal  S_{w_1}$,  whereas the indexes $R_1+1,...,R_1+R_2$ refer to the $(m,k)$-stars in
$\mathcal  S_{w_2}$, and so on.

We focus on the $R_i$ $(m,k)$-stars in $\mathcal  S_{w_i}$.
The rows in $A$ corresponding to the $k_j$-pendant vertices with$j\in\{\sum_{q=1}^{i-1} R_q+1,...,\sum_{q=1}^{i} R_q\}$,
are $m_j$ vectors $(v^{(j)}_{j_1},...,v^{(j)}_{j_{m_j}})$, linearly dependent and such that $v^{(j)}_{j_1}=...=v^{(j)}_{j_{m_j}}$,
whose indexes are
$$j_1=\sum_{p=1}^{j-1} m_{p}+1,...,{j_{m_j}}=\sum_{p=1}^{j-1} m_{p}+m_j$$
when $j>1$, or
$$j_1=1,...,{j_{m_j}}=m_j$$
when $j=1$.\\
Then we get
$$v^{(j)}_{j_1},...,v^{(j)}_{j_{{m_j}-1}}\in ker(A),\quad \forall j\in\{\sum_{q=1}^{j-1} R_q+1,...,\sum_{q=1}^{j} R_q\}.$$
and
$$\exists \mu_{j_1},...,\mu_{j_{{m_j}-1}}\in\sigma(A)\quad \mbox{ such that }\quad \mu_{j_1}=...=\mu_{j_{{m_j}-1}}=0.$$

This is true for each $j\in\{\sum_{q=1}^{j-1} R_q+1,...,\sum_{q=1}^{j} R_q\}$, so that
$$\exists \mu_1,...,\mu_{deg(\mathcal S_{w_i})} \in\sigma(A)\quad \mbox{ such that }\quad \mu_1=...=\mu_{deg(\mathcal S_{w_i})}=0.$$

and the Laplacian matrix $L$ has at least $deg(\mathcal S_{w_i})+R_i$ diagonal entries with value $w_i$.\\

In the matrix $(L-w_i I)$ there are $v^{(j)}_{j_q}, \ q\in\{1,...,m_j\}$ vectors linearly dependent for each $j$, as a consequence
 $v^{(j)}_{j_1},...,v^{(j)}_{j_{m_j-1}}\in ker(L-w_i I)$ and
$$\exists \mu_1,...,\mu_{deg(\mathcal S_{w_i})} \in\sigma(L-w_i I)\quad \mbox{ such that }\quad \mu_1=...=\mu_{deg(\mathcal S_{w_i})} =0.$$

Finally, let $\mu_p$ be one of these eigenvalues, then

$$0=det((L-w_i I)-\mu_p I)=det(L-(w_i +\mu_p )I)$$

and $\lambda:=w_i\in\sigma (L)$ with multiplicity greater or equal to $deg(\mathcal S_{w_i})$.\\

\end{proof}
}

Some corollaries on the signless and normalized Laplacian matrices can be obtained by using similar proofs.
Let $B$ and $\hat L$ be the signless and normalized Laplacian matrices of $\mathcal G=(\mathcal V,\mathcal E,w)$ respectively
and let $\sigma(B)$, $\sigma(\hat L)$ the eigenvalues of $B$ and $\hat L$ with algebraic multiplicity
$m_B(\lambda)$, $m_{\hat L}(\lambda)$ for the eigenvalue $\lambda$ in $B$ and $\hat L$ respectively.\\
\begin{cor}
If
\begin{itemize}
\item $s$ is the number of all the $S_{m,k}$ as $m$ and $k$ vary in $\mathbb N$  and $m+k\leq n,$ of $\mathcal G$,
\item $r$ is the number of $S_{m,k}$ with different weights, $w_1,...,w_r$,\\
\end{itemize}
then for any $i\in\{1,...,r\},$
$$\exists \lambda\in{\sigma(B)} \mbox{ such that } \lambda=w_i \mbox{ and } m_B(\lambda)\geq deg(\mathcal S_{w_i})$$
where $\mathcal S_{w_i}:=\{S_{m,k}\in \mathcal G | w(S_{m,k})=w_i\}$.
\end{cor}

\begin{cor}
If
\begin{itemize}
\item $s$ is the number of all the $S_{m,k}$ as $m$ and $k$ vary in $\mathbb N$  and $m+k\leq n,$ of $\mathcal G$,
\item $r$ is the number of $S_{m,k}$ with different weights, $w_1,...,w_r$,\\
\end{itemize}
then for any $i\in\{1,...,r\},$
$$\exists \lambda\in{\sigma(\hat L)} \mbox{ such that } \lambda=1 \mbox{ and } m_{\hat L}(\lambda)\geq \sum_{i=1}^r deg(\mathcal S_{w_i})$$
where $\mathcal S_{w_i}:=\{S_{m,k}\in \mathcal G | w(S_{m,k})=w_i\}$.
\end{cor}

A wider class of graphs for which the previous results can be extended is the class of the $l$-dependent graphs, defined as follows:

\begin{defn}[$l$-dependent graph: $D^l$]
A $l$-dependent graph is a graph $(\mathcal V,\mathcal E, w)$ whose vertices can be partitioned into four subsets:
the independent set $\mathcal V_1$, the central set $\mathcal V_2$, the independent set $\mathcal V_3$ and the set
$\mathcal V\setminus (\mathcal V_1\cup\mathcal V_2 \cup\mathcal V_3)$ such that

\begin{enumerate}
\item each vertex in $\mathcal V_1$ has at least one edge in $\mathcal V_2$ and vice versa, i.e.
$$\forall i\in \mathcal V_1,\exists j\in \mathcal V_2\  \mbox{ such that } \  (i,j)\in \mathcal E$$
 $$\forall j\in \mathcal V_2,\exists i\in \mathcal V_1\  \mbox{ such that } \  (i,j)\in \mathcal E$$
\item vertices in $\mathcal V_1$ and $\mathcal V_3$ have edges only in $\mathcal V_2$, i.e.
$$\forall i\in \mathcal V_1\cup\mathcal V_3,\forall j\in \mathcal V\setminus\mathcal V_2, \quad (i,j)\notin \mathcal E$$
\item vertices in $\mathcal V_3$ have only edges that are a linear combination of all the edges of some vertices in $\mathcal V_1$, i.e.
$$\forall i\in\mathcal V_3,  \exists j_1,...,j_{l_i}\in \mathcal V_1 \mbox{ such that }$$
$$\forall j\in\{j_1,...,j_{l_i}\}, \  \forall z \ \mbox{ such that } \  (j,z)\in\mathcal E, z\in\mathcal V_2\Rightarrow $$
$$\exists a(j)\in \mathbb R^{> 0} \mbox{ and }  (i,z)\in\mathcal E, \ \mbox{ such that }\  w(i,z)=a(j)w(j,z).$$
\item $\mathcal V_1, \mathcal V_2, \mathcal V_3\subseteq \mathcal V$ are kept in order to satisfy the following
condition $$l:=\max_{\mathcal V_1,\mathcal V_2, \mathcal V_3\subseteq \mathcal V}|\mathcal V_3|.$$
\end{enumerate}

A $l$-dependent graph with $|\mathcal V_3|=l$, is denoted $D^l.$
\end{defn}

\begin{figure}[!!h]
\centering
\includegraphics[width=6cm]{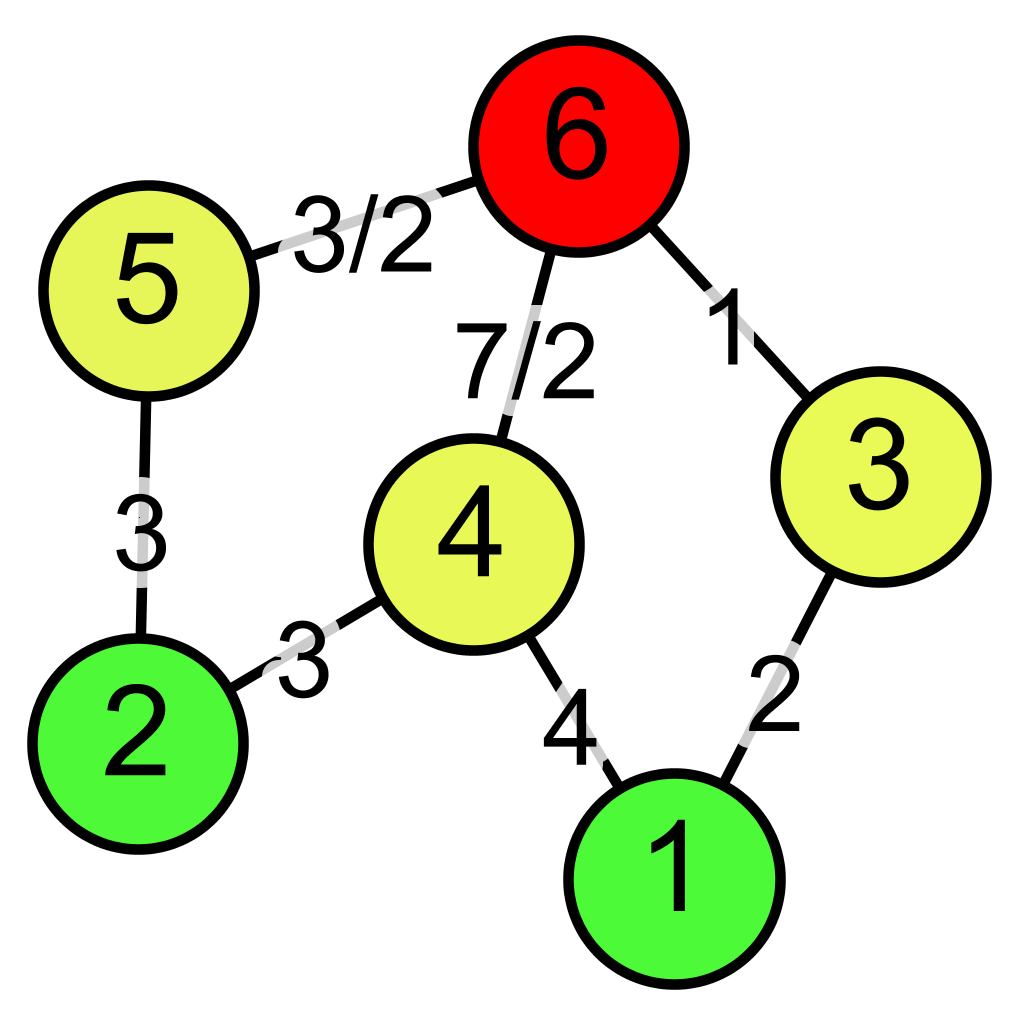}\caption{$D^l(\tilde w)$ graph, where the subsets $\mathcal V_1$
(for example the green vertices), $\mathcal V_2$ (the yellow vertices), $\mathcal V_3$ (for example the red vertex) and
$\mathcal V\setminus (\mathcal V_1\cup\mathcal V_2 \cup\mathcal V_3)$ are respectively with cardinality $\bar m=\underline m=2$,
$\bar k=\underline k=3$, $l=1$ and $|\mathcal V\setminus (\mathcal V_1\cup\mathcal V_2 \cup\mathcal V_3)|=0$. In the Laplacian matrix
there is the eigenvalue $\lambda=\tilde w=6$ with multiplicity 1.}
\end{figure}

\begin{obs}
First of all, we remark that neither the uniqueness of partition nor the cardinality of both $\mathcal V_1$ and $\mathcal V_2$ sets is guaranteed.
If we require the uniqueness of the cardinality further conditions are necessary: for instance
\begin{enumerate}
\item[5.*] maximum cardinality of the sets $\mathcal V_1,\mathcal V_2$
$$\bar{m}:=\max_{\mathcal V_1,\mathcal V_2 \subseteq \mathcal V\setminus\mathcal V_3}|\mathcal V_1|$$
$$\bar k:=\max_{\mathcal V_1,\mathcal V_2 \subseteq \mathcal V\setminus\mathcal V_3}|\mathcal V_2|$$
\item[5.**] minimum cardinality of the sets $\mathcal V_1,\mathcal V_2$
$$\underline m:=\min_{\mathcal V_1,\mathcal V_2 \subseteq \mathcal V\setminus\mathcal V_3}|\mathcal V_1|$$
$$\underline k:=\min_{\mathcal V_1,\mathcal V_2 \subseteq \mathcal V\setminus\mathcal V_3}|\mathcal V_2|.$$
\end{enumerate}

Even by requiring the maximum or minimum cardinality of both $\mathcal V_1$ and $\mathcal V_2$ sets, the uniqueness of the partition is not univocally determined.\\
The uniqueness of the set $\mathcal V_2$ is satisfied whenever one of the conditions 5. holds.
We notice that according to 5.**, the set $\mathcal V_2$ is defined as the set of all the
vertices $i\in\mathcal V$ such that $(i,j)\in\mathcal E,\ j\in\mathcal V_3.$
\end{obs}

\begin{figure}[!!h]
\centering
\includegraphics[width=6cm]{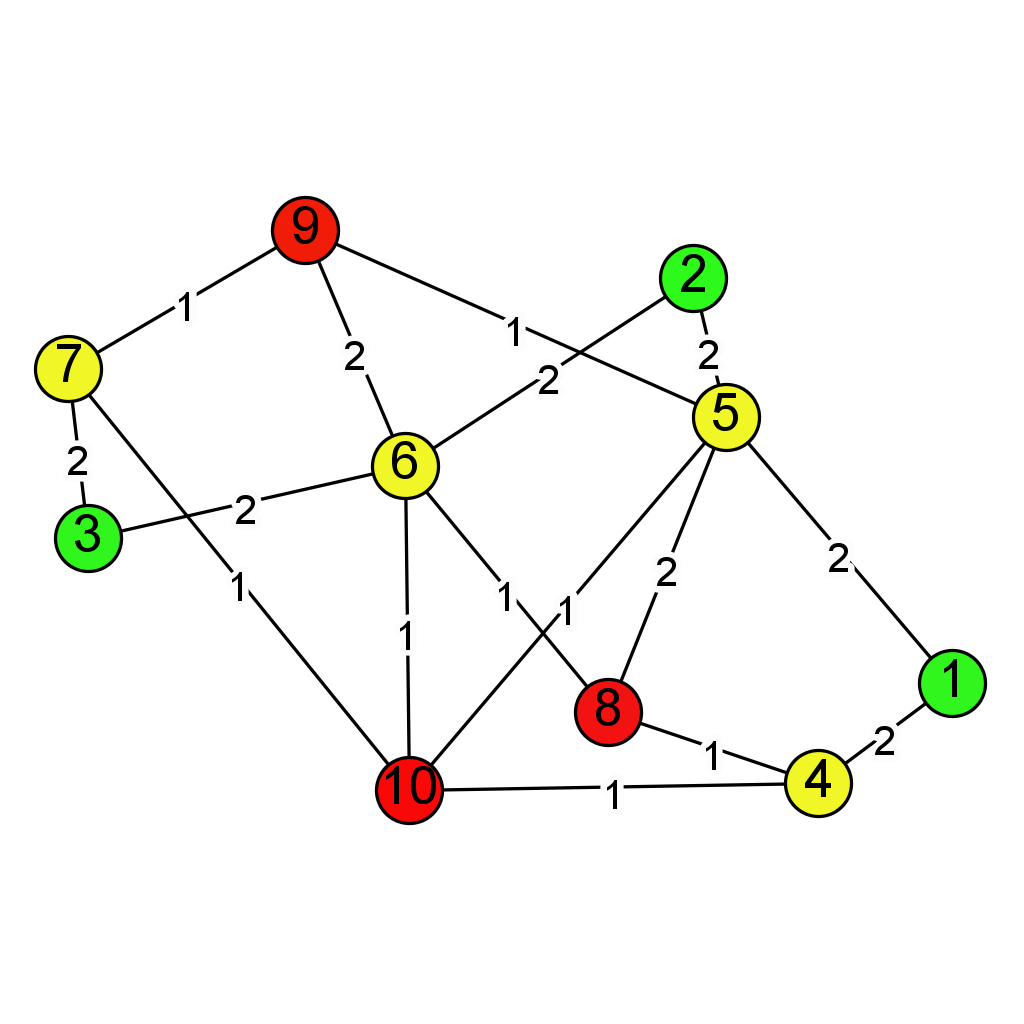}
\includegraphics[width=6cm]{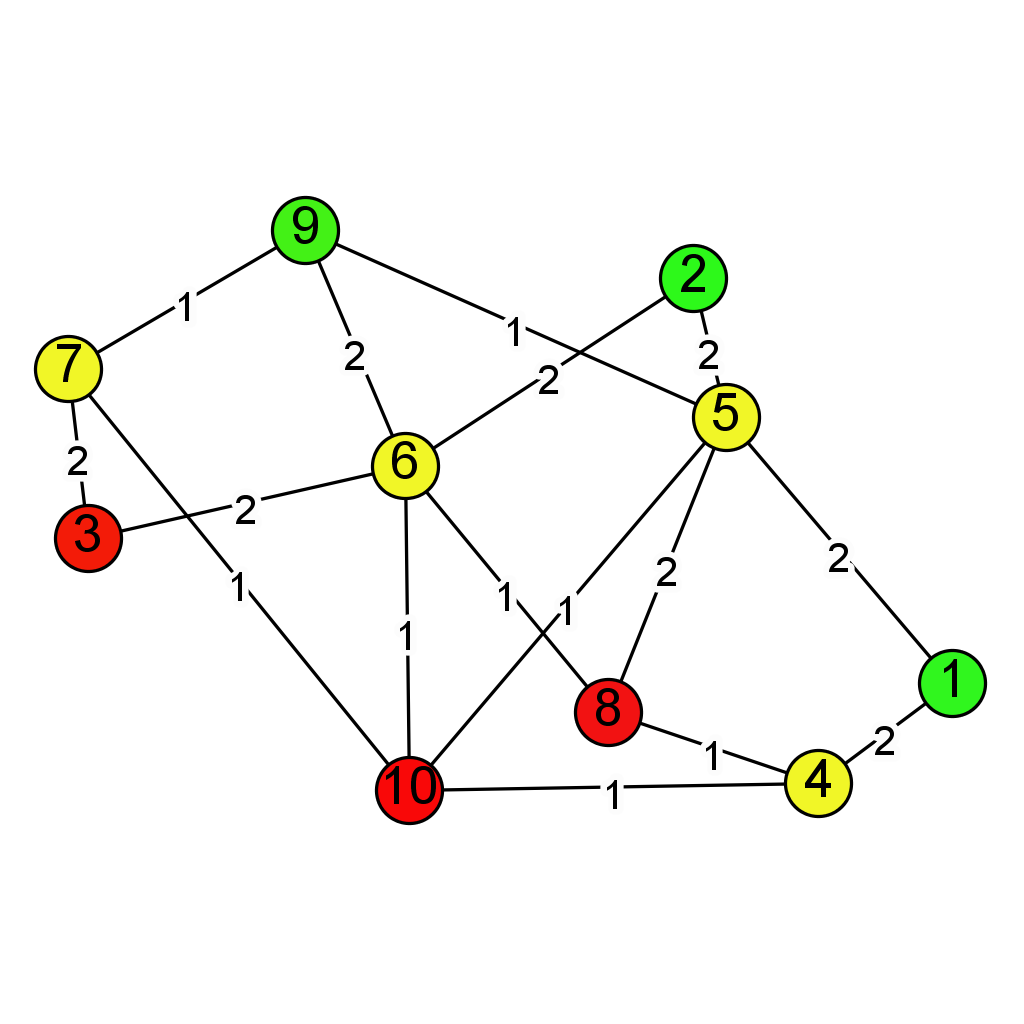}
\caption{$D^l(\tilde w)$ graph, where the subsets $\mathcal V_1$ (green vertices) and $\mathcal V_3$ (red vertices) can be chosen differently.
The cardinalities of the sets are respectively $\bar m=\underline m=3$, $\bar k=\underline k=4$, $l=3$ and
$|\mathcal V\setminus (\mathcal V_1\cup\mathcal V_2 \cup\mathcal V_3)|=0$.
In the Laplacian matrix there is the eigenvalue $\lambda=\tilde w=4$ with multiplicity 3.}
\end{figure}

\begin{obs}

Whenever in the condition [3.] the set $\{j_1,...,j_{l_i}\}$ coincides with the set $\mathcal V_1$,
 then the $l$-dependent graph is also a graph with an (m,k)-star, with m=l+1.
\end{obs}

We define an $l$-dependent graph of weight $\tilde w$, $D^l(\tilde w)$ as the $l$-dependent graph such that each vertex $i\in\mathcal V_1\cup\mathcal V_3$
has strength $\tilde w$.\\
Now we can extend the Theorem \ref{th:one} on graphs with $(m,k)$-star to $l$-dependent graphs, that is one of the main results of this work.\\

Let $\mathcal G=(\mathcal V,\mathcal E,w)$ be a graph, and $L$ the Laplacian matrix of $\mathcal G$.

\begin{thm}\label{th:main}
If $\mathcal G$ be a $D^l(\tilde w)$ graph, with $\tilde w\in\mathbb(R^{>0})$ and $l\in \mathbb N$,\\
then
$$\exists \lambda\in{\sigma(L)} \mbox{ such that } \lambda=\tilde w \mbox{ and } m_L(\lambda)\geq l.$$
\end{thm}

\begin{proof}
The proof is similar to Theorem \ref{th:one}. By definition of $D^l(\tilde w)$, each vertex $i\in\mathcal V_3$ has a corresponding
row in the adjacency matrix $A$, that is a linear combination of the rows of some vertices $j_1,...,j_{l_i}\in\mathcal V_1$.
Therefore the adjacency matrix $A$ has an eigenvalue $\mu=0$ of multiplicity at least $l$.
Since each vertex $i\in\mathcal V_1\cup\mathcal V_3$ has strength $\tilde w$ we can conclude the proof.
\end{proof}

\begin{obs}
The previous result does not require the conditions 5.
\end{obs}
We observe that a $D^{l}(\tilde w)$ graph, with $l\in\mathbb{N},\ \tilde w\in\mathbb{R}^+$, could be also a $D^{l_*}(\tilde w_*)$ graph,
for any $\ l_*\in\mathbb{N},\ \tilde w_*\in\mathbb{R}^+$.

As for the Theorem \ref{th:one}, some corollaries on the signless and normalized Laplacian matrices can be obtained by means of similar proofs.
Let $\mathcal G=(\mathcal V,\mathcal E,w)$ be a graph, and $B$ and $\hat L$ the signless and normalized Laplacian matrices respectively.
\begin{cor}\label{cor:main}
If
$\tilde w_1,...,\tilde w_m\in\mathbb(R^{>0})$ and $l_1,...,l_m\in \mathbb N$ such that $\mathcal G$ is a $D^l_i(\tilde w_i)$ graph, $ i\in\{1,...,m\}$;\\
then
$$\exists \lambda \in{\sigma(\hat L)} \mbox{ such that } \lambda=1 \mbox{ and } m_{\hat L}(\lambda)\geq \sum_{i=1}^m l_i.$$
\end{cor}


\subsection{(m,k)-star graph reduction}
According to the previous results, we have defined a class of graphs whose Laplacian matrices have an eigenvalues spectrum with known multiplicities and values.
Now, our aim is to simplify the study of such graphs by collapsing these vertices into a single vertex replacing the original graph with a reduced graph.
\\
At this purpose, the following definitions are useful:
\begin{defn}[$(m,k)$-star $q$-reduced: $S^q_{m,k}$]
A $q$-reduced $(m,k)$-star is a $(m,k)$-star of vertex sets $\{\mathcal V_1,\mathcal V_2\}$, such that $q$ of its vertices in $\mathcal V_1$ are removed.
Hence the order and degree of the $S^q_{m,k}$ are $m+k-q$ and $m-q-1$ respectively.\\
\end{defn}

\begin{defn}[$q$-reduced graph: $\mathcal G^q$]
A $q$-reduced graph $\mathcal G^q$ is obtained from a graph $\mathcal G$ with some $(m,k)$-stars removing $q$ of the vertices in the set $\mathcal V_1$
of $\mathcal G.$
\end{defn}

We derive a spectrum correspondence between  graphs $\mathcal G$ and $\mathcal G^q$

\begin{defn}[Mass matrix of a $S_{m,k}^q$]
Let $\mathcal V_1$ and $\mathcal V_2$ be the vertex sets of the graph $S_{m,k}^q,\ q< m$.\\
Let $B$ be the adjacency matrix of  $S_{m,k}^q$. The mass matrix of a $S_{m,k}^q$, $M$ is a diagonal matrix of order $m+k-q$  such that

\begin{equation}\label{eq:diagM}
M_{ii} = \begin{cases} \frac{m}{m-q}, & \mbox{if } i\in\mathcal V_1 \\ 1 & \mbox{otherwise } \end{cases},
\end{equation}
\end{defn}
The mass matrix $M$ can be defined in the same way also for a graph $\mathcal G^q$, with one (or more) $S_{m,k}^q$ by means
of a matrix of order $n-q$, whenever the graph $\mathcal G^q$ is composed by $n-q$ vertices.\\

Now we state the second main result of this paper.

\begin{thm}[$(m,k)$-star adjacency matrix reduction theorem]\label{th:reduction1}
Let
\begin{itemize}
\item $\mathcal G$ be a graph, of n vertices, with a $S_{m,k}, \ m+q\leq n$,
\item $\mathcal G^q$ be the reduced graph with a $S_{m,k}^q$ instead of $S_{m,k}$, of $n-q$ vertices,\\
\item $A$ be the adjacency matrix of $\mathcal G$,
\item $B$ be the adjacency matrix of $\mathcal G^q$,
\item $M$ be the diagonal mass matrix of $\mathcal G^q$,
\end{itemize}
then
\begin{enumerate}
\item $\sigma(A)=\sigma(MB),$
\item There exists a matrix $K\in\mathbb R^{n\times (n-q)}$ such that $M^{1/2}BM^{1/2}=K^TAK$ and $K^TK=I$. Therefore, if $v$ is an eigenvector of $M^{1/2}BM^{1/2}$ for an
eigenvalue $\mu$, then Kv is an eigenvector of A for the same eigenvalue $\mu$.
\end{enumerate}
\end{thm}
{ Before proving Theorem \ref{th:reduction1}, we recall the well known result for eigenvalues of
symmetric matrices, \cite{Hwang2004}.

\begin{lem}[Interlacing theorem]
Let $A\in Sym_n(\mathbb R)$ with eigenvalues $\mu_1(A)\geq...\geq \mu_n(A).$ For $m<n$, let $S\in\mathbb R^{n,m}$ be a
matrix with orthonormal columns, $K^TK=I$, and consider the  $B=K^TAK$ matrix, with eigenvalues $\mu_1(B)\geq...\geq \mu_m(B).$
If
\begin{itemize}
\item the eigenvalues of $B$ interlace those of $A$, that is,
$$\mu_i(A)\geq\mu_i(B)\geq\mu_{n_A-n_B+i}(A), \quad i=1,...,n_B,$$
\item if the interlacing is tight, that is, for some $0\leq k\leq n_B,$
$$\mu_i(A)=\mu_i(B), \ i=1,...,k\ \mbox{ and } \ \mu_i(B)=\mu_{n_A-n_B+i}(A), \ i=k+1,...,n_B$$
then $KB=AK.$
\end{itemize}
\end{lem}

\begin{proof}
First we prove the existence of the $K$ matrix:\\
let $\mathcal P=\{P_1,...,P_{n_B}\}$ be a partition of the vertex set $\{1,...,n_A\}$, where $n_B=n_A-q.$
The \textit{ characteristic matrix H } is defined as the matrix where the $j$-th column is the characteristic vector of $P_j$ ($j=1,...,n_B$).\\
Let A be partitioned according to $\mathcal P$
\[
A=\left(
\begin{array}{ccc}
A_{1,1} & \dots & A_{1,n_B} \\
\vdots &  & \vdots \\
A_{n_B,1} & \dots & A_{n_B,n_B}
\end{array}
\right),\]

where $A_{i,j}$ denotes the block with rows in $P_i$ and columns in $P_j$.
The matrix $B=(b_{ij})$ whose entries $b_{ij}$ are the averages of the $A_{i,j}$ rows, is called the \textit{quotient matrix} of $A$ with respect $\mathcal P$,
i.e. $b_{ij}$ denote the average number of neighbours in $P_j$ of the vertices in $P_i$.\\
The partition is equitable if for each $i,j$, any vertex in $P_i$ has exactly $b_{ij}$ neighbours in $P_j$.
In such a case, the eigenvalues of the quotient matrix $B$ belong to the spectrum of $A$ ($\sigma(B)\subset\sigma(A)$) and the spectral radius
of $B$ equals the spectral radius of $A$: for more details cfr. \cite{brouwer12}, chapter 2.\\
Then we have the relations
$$MB=H^TAH, \quad H^TH=M.$$

Considering a $q$-reduced $(m,k)-$star with adjacency matrix $B$, we weight it by a diagonal mass matrix $M$ whose diagonal entries
are one except for the $m-q$ entries of the vertices in $\mathcal V_1$,

\begin{equation}\label{eq:diagM}
M_{ii} = \begin{cases} \frac{m}{m-q}, & \mbox{if } i\in\mathcal V_1 \\ 1 & \mbox{otherwise } \end{cases},
\end{equation}

and we get
$$MB\sim M^{1/2}BM^{1/2}=K^TAK, \quad K^TK=I,$$

where $K:=HM^{1/2}.$
In addition to the th.(\ref{th:one}), the eigenvalues of $MB$ are a subset of the eigenvalues of $A$,
the adjacency matrix of the corresponding $S_{m,k}$ graph

$$
\sigma(MB)\subset\sigma(A).
$$

Whenever $q<m-1$, we get $\sigma(MB)=\sigma(A)$, up to the multiplicity of the eigenvalue $\mu=0$.\\

Finally, if $v$ is an eigenvector of $M^{1/2}BM^{1/2}$ with eigenvalue $\mu$, then $Kv$ is an eigenvector of $A$ with the same eigenvalue $\mu$.\\

Indeed form the equation
$$\tilde Bv=\mu v$$
an taking into account that the partition is equitable, we have $K\tilde B=AK,$ and
$$AKv=K\tilde Bv=\mu Kv.$$

\end{proof}}

We obtain a similar result for the Laplacian matrix.

\begin{thm}[$(m,k)$-star Laplacian matrix reduction theorem]\label{th:reduction2}
If
\begin{itemize}
\item $\mathcal G$ be a graph, of n vertices, with a $S_{m,k}, \ m+q\leq n$,
\item $\mathcal G^q$ be the reduced graph with a $S_{m,k}^q$ instead of $S_{m,k}$, of $n-q$ vertices,
\item $L(A)$ be the Laplacian matrix of $\mathcal G$,
\item $L(B)$ be the Laplacian matrix of $\mathcal G^q$. Let $M$ the diagonal mass matrix of $\mathcal G^q$,
\end{itemize}
then
\begin{enumerate}
\item $\sigma(L(A)=\sigma(L(MB))$
\item There exists a matrix $K\in\mathbb R^{n\times (n-q)}$ such that $M^{1/2}BM^{1/2}=K^TAK$ and $K^TK=I$.
Therefore, if $v$ is an eigenvector of $\tilde L(M B):=diag(MB)-M^{1/2}BM^{1/2}$ for an eigenvalue $\lambda$, then Kv is
an eigenvector of L(A) for the same eigenvalue $\lambda$.
\end{enumerate}
\end{thm}
{The proof for the Laplacian version of the Reduction Theorem \ref{th:reduction1} is similar to that for the adjacency matrix,
in fact using the same arguments as in the proof of
\ref{th:reduction1}, we can say that 1. is true and that the $K$ matrix exists. So we prove directly only the second part of  point 2. of the theorem.
\begin{proof}
Let $v$ be an eigenvector of $L(\tilde B):=diag(MB)-M^{1/2}BM^{1/2}$ for an eigenvalue $\lambda$, then
$$L(\tilde B)v=\lambda v.$$
Because of $K\tilde B=AK$ and $diag(A)K=Kdiag(MB)$, we obtain
$$L(A)Kv=diag(A)Kv-AKv=Kdiag(MB)v-K\tilde B v=\lambda Kv.$$
\end{proof}}

According to the previous results, graphs with $(m,k)$-stars and graphs $q$-reduced can be partitioned in the same way, up to the removed vertices.\\

\begin{cor}\label{cor:reduction}
Under the hypothesis of theorem \ref{th:reduction2}, if $v$ is a (left or right) eigenvector of $L(MB)$ with eigenvalue $\lambda$,
then its entries have the same signs of the entries of the eigenvector $u$ of $L(A)$ with the same eigenvalue $\lambda$. \end{cor}

Indeed, the matrices $L(MB)$ and $\tilde L(MB)$ are similar, by means of the non singular matrix $M^{1/2}$.
Furthermore, since the similarity matrix  $M^{1/2}$ is diagonal with all positive elements on the diagonal,
then both  left and right eigenvectors of $ L(MB)$ preserve the sign of the eigenvectors of $\tilde L(MB)$. {  We formally prove the Corollary.

\begin{proof} 
$\tilde L(MB)$ and $L(MB)$ are similar by means of the matrix $M^{1/2}$, in fact
\begin{eqnarray}
M^{-1/2}L(MB)M^{1/2}&=&M^{-1/2}diag(MB)M^{1/2}-M^{-1/2}MBM^{1/2}\nonumber\\
&=&diag(MB)-M^{1/2}BM^{1/2}\nonumber\\
&=&\tilde L(MB). \nonumber
\end{eqnarray}

$L(MB)$ preserves the sign of the eigenvectors of $\tilde L(MB)$.\\
If $\tilde v$ an eigenvector of $\tilde L(MB)$ of the eigenvalue $\lambda\in\sigma(\tilde L(MB))$, then
\begin{eqnarray}
\tilde L(MB) \tilde v=\lambda \tilde v & \Leftrightarrow & M^{-1/2}L(MB)M^{1/2} \tilde v=\lambda \tilde v\nonumber\\
& \Leftrightarrow & L(MB)M^{1/2} \tilde v=\lambda M^{1/2} \tilde v\nonumber
\end{eqnarray}
As a consequence $v:=M^{1/2} \tilde v$ is the eigenvector of $L(MB)$ of the eigenvalue $\lambda,$ and $v_i=(M\tilde v)_i$,
$$v_i=\sum_{r=1}^{n-q} M_{ir}\tilde v_r=M_{ii}\tilde v_i.$$

\end{proof}}
Thanks to the previous result, we can partition the primary graph $\mathcal G$ containing the $(m,k)$-star and the $q$-reduced graph $\mathcal G^q$, weighted
by the matrix $M$, in the same way except for the removed vertices.\\

\section{Concluding remarks}\label{sec:4}

In this work we considered the problem of spectral partitioning of weighted graphs that contain $(m,k)$-stars.
We showed that, under some hypotheses on edge weights, the Laplacian matrix of graphs with $(m,k)$-stars has eigenvalues of multiplicity at least $m-1$
and computable values.\\
We proved that it is possible to reduce the node cardinality of these graphs by a suitable equivalence relation, keeping the same eigenvalues
on the adjacency and Laplacian matrices up to their multiplicity.\\
Furthermore, we have shown that Laplacian matrices of both the original and reduced graphs have the same signs of the eigenvectors entries,
so that it is possible to partition both  graphs in the same way, up to removed vertices.\\
According to these results, whenever a weighted graph is composed by one or more $(m,k)$-star subgraphs, it is possible to collapse some of its vertices
into one, and to reduce the dimension of the matrices associated to these graphs, preserving the spectral properties.\\
These results can be relevant for applications to the network partitioning problems, or whenever a sort of node summarization is sought, merging nodes
with similar spectral properties.
These nodes could share similar functional properties, e. g. in the case of proteins with a similar neighborhood structure in interactome networks\cite{bioplex15},
with implications on biomedical and Systems Biology applications \cite{menche15}.
Moreover, the possibility to reduce network dimensionality by an equivalence relation among nodes can possibly be extended in a perturbative approach,
performing network reduction whenever the conditions of our theorems are 'almost satisfied', that is if some eigenvalues are sufficiently close.

\section{Acknowledgments}
The authors thank Nicola Guglielmi (University of L'Aquila, Italy), and E. A. also thanks
Domenico Felice (Max Planck Institute of Leipzig, Germany) and Carmela Scalone (University of L'Aquila, Italy) for useful discussions.

\bibliography{biblio_coal}
\bibliographystyle{plain}

\end{document}